\documentclass[10pt]{amsart}

\usepackage{amssymb,amsfonts,amsthm,amscd,stmaryrd}

\usepackage{ae}
\usepackage[utf8]{inputenc}
\numberwithin{figure}{section}

\usepackage{mathrsfs,a4wide}
\usepackage{esint}

\usepackage{import}

\usepackage[leqno]{amsmath}

\usepackage[final,allcolors=blue,colorlinks=true]{hyperref}

\numberwithin{figure}{section}

\newtheorem{theorem}{Theorem}[section]
\newtheorem{lemma}[theorem]{Lemma}
\newtheorem{proposition}[theorem]{Proposition}
\newtheorem{corollary}[theorem]{Corollary}
\theoremstyle{definition}
\newtheorem{definition}[theorem]{Definition}

\newtheorem{remark}[theorem]{Remark}
\numberwithin{equation}{section}

\newcommand{\R}{\mathbb{R}}
\newcommand{\N}{\mathbb{N}}

\newcommand{\dist}{{\rm dist}}
\newcommand{\eps}{\varepsilon}
\newcommand{\vphi}{\varphi} 

\newcommand{\Pu}{{\mathcal{P}}}

\begin{document}

\title[Generalized Harnack inequality for nonhomogeneous elliptic equations]{Generalized Harnack  inequality for nonhomogeneous elliptic equations}

\author{Vesa Julin}


\begin{abstract} This paper is concerned with  nonlinear elliptic  equations in nondivergence form 
\[
F(D^2u, Du, x) = 0
\]
where  $F$  has  a drift term which is not  Lipschitz continuous. Under this condition  the equations are nonhomogeneous and   nonnegative solutions  do not satisfy the  classical  Harnack inequality. This paper presents a new  generalization of Harnack inequality for such equations. As a corollary we obtain  the optimal  Harnack type of  inequality for $p(x)$-harmonic functions which quantifies the strong minimum principle.
\end{abstract}

\maketitle


\section{Introduction}
The famous Krylov-Safonov theorem \cite{KS1}, \cite{KS2}  states that nonnegative solution $u \in C(B_{2R}(x_0))$ of a linear,  uniformly elliptic  equation  
\[
\text{Tr}(A(x)D^2u) = 0
\]
with measurable and bounded coefficients satisfies the Harnack inequality 
\begin{equation} \label{cl_harnack}
\sup_{B_{R}(x_0)} u \leq C  \inf_{B_{R}(x_0)} u 
\end{equation}
where $C$ is a universal constant. This result is important since it quantifies  the strong minimum principle (SMP) and  gives H\"older estimate.   There are numerous generalizations of this result e.g. by Trudinger   \cite{Tr},\cite{GT} for quasilinear operators and by Caffarelli  \cite{C1},\cite{CC} for  fully nonlinear operators.  

In this paper we consider  elliptic equations in nondivergence form
\begin{equation} \label{thePDE}
F(D^2u, Du, x) = 0
\end{equation}
which are nonhomogeneous. The operator $F$  is assumed to be elliptic in the sense that there are  $0< \lambda \leq \Lambda$ such that 
 \[ 
\lambda \text{Tr}(Y) \leq F(X, p, x)  -  F(X +Y, p, x)\leq   \Lambda \text{Tr}(Y)
\]
for every symmetric matrices  $X, Y$ where $Y$ is positive semidefinite,  and for every $(x,p) \in B_{2R}(x_0) \times \R^n$. We assume that $F$ has a drift term which satisfies growth condition
\[
|F(0, p, x)  \big| \leq \phi(|p|)
\]
for every  $(x,p) \in B_{2R}(x_0) \times \R^n$  where $\phi: [0,\infty) \to  [0,\infty)$ is continuous and increasing. The function $\phi$  thus describes how far the equation is from being homogeneous. 

In the sublinear case $\phi(t) \leq  t$  the equation \eqref{thePDE} is essentially  homogeneous and nonnegative  solutions  satisfy   \eqref{cl_harnack},  see \cite{Sa}. In the superlinear case $\phi(t)>t$ the equation \eqref{thePDE} is  nonhomogeneous and it is  known that  the estimate \eqref{cl_harnack} does not hold for a uniform constant $C$.  In the subquadratic case $\phi(t) \leq t^2$  it was proven in \cite{Tr2}  (see also \cite{KT}) that  nonnegative solution $u$ of \eqref{thePDE} satisfies
\[
\sup_{B_{R}(x_0)} u \leq C ( \inf_{B_{R}(x_0)} u + R^2)
\]
where the constant $C$ depends on ${L^\infty}$-norm of $u$. Although this estimate is strong enough  to give H\"older continuity  it has the disadvantage that it does not imply SMP. Moreover, in order to use this estimate one has to know a priori  that the solution  is bounded. 

We introduce a new type of  Harnack  inequality for nonegative solutions of \eqref{thePDE}. This inequality is a natural  generalization of \eqref{cl_harnack} since it quantifies SMP (whenever it is true) in a sharp way, and the constant  in the  inequality is  independent of the solution itself. The novelty of this Harnack inequality is that it depends on $\phi$, i.e., it depends on the scaling of the equation \eqref{thePDE}.  We make some assumptions on the asymptotic behaviour of $\phi$. Under these conditions the inequality  gives  H\"older continuity  for the solutions.

The quantification of  SMP turns out to be  a rather delicate issue. A naive example shows that if  $\phi$ is merely  H\"older continuous,  $\phi(t)= t^{\alpha}$  for $t \in [0,1]$ and  $\alpha < 1$, SMP fails. In fact, SMP can only be true if $\phi$ satisfies the  Osgood condition
\begin{equation}
\label{osgood}
 \int_{0}^{\eps} \frac{dt}{\phi(t)}  = \infty  \qquad \text{for every } \eps >0.
\end{equation}
Indeed, this can be easily  seen by defining  $v:(-2R, 2R): \to [0, \infty)$ such that $v(x)= 0$ for $x \leq 0$ and 
\[
x = \int_0^{v(x)}\frac{dt}{\phi(t)} \qquad \text{for } \,  x \geq 0.
\]
The function $u(x)= \int_{0}^x v(t)\, dt$ then satisfies $u''= \phi(u')$ and violates SMP. Hence, if  we require  our generalized Harnack inequality to  quantify SMP, we have to take into account the  condition \eqref{osgood}. This suggests that the inequality has to be in integral form.

Let us now state precisely the main result.  Following the idea of Caffarelli \cite{C1}  we replace the equation \eqref{thePDE}  by two inequalities which follow from the ellipticity condition and from the growth condition of the drift term. In other words we assume that $u \in C(B_{2R}(x_0))$ is  a viscosity  supersolution of 
\begin{equation}
\label{model1}
\Pu_{\lambda, \Lambda}^+(D^2 u)  \geq -\phi(|Du|)  
\end{equation}
and  a viscosity  subsolution of 
 \begin{equation}
\label{model2}
\Pu_{\lambda, \Lambda}^-(D^2 u)  \leq \phi(|Du|) 
\end{equation}
in $B_{2R}(x_0)$.   Here $\Pu_{\lambda, \Lambda}^-, \Pu_{\lambda, \Lambda}^+$  are the  usual Pucci operators defined in the next section. The function $\phi:[0, \infty) \to [0, \infty)$ is assumed to be of the form $\phi(t)= \eta(t)t$ and to satisfy the following conditions: 
\begin{itemize}
\item[(P1)]  $\phi: [0, \infty) \to  [0, \infty)$  is increasing, locally  Lipschitz continuous in $(0, \infty)$ and $\phi(t)\geq t$ for every $t \geq 0$. Moreover,   $\eta: (0, \infty) \to  [1,\infty)$  is nonincreasing on $(0,1)$ and nondecreasing on $ [1, \infty) $.
\item[(P2)]  $\eta$ satisfies
\[
\lim_{t \to \infty} \frac{t \eta'(t)}{\eta(t)} \log(\eta(t)) = 0.
\]
\item[(P3)] There is a constant $\Lambda_0$ such that
\[
\eta(st) \leq \Lambda_0 \eta(s) \eta(t).
\] 
for every $s,t \in (0, \infty)$.
\end{itemize}
Roughly speaking (P2) means that $\eta$ is slowly increasing function. Note that this  is  a condition for the asymptotic behaviour of $\phi$. In particular, we do not assume that $\phi$ satisfies \eqref{osgood}. We  say that constant is universal if it depends only on $\lambda, \Lambda, \phi$ and the dimension of the space.

Here is the main result of the paper. 
 \begin{theorem}
\label{weakHarnack}
Let   a nonnegative function $u\in C(B_{2R}(x_0))$, for $R \leq 1$, be a viscosity  supersolution of \eqref{model1} and a  viscosity subsolution of \eqref{model2} in $B_{2R}(x_0)$. Denote $m:= \inf_{B_R(x_0)}u$ and  $M:= \sup_{B_R(x_0)}u$. There is a  universal  constant $C$  such that 
\begin{equation}
\label{osgood_harnack}
\int_{m}^{M} \frac{dt}{R^2\phi(t/R)+ t} \leq C.
\end{equation}
\end{theorem}

We immediately note that, using the notation of the previous theorem,  \eqref{osgood_harnack} implies the estimate
\[
\int_{m}^{M} \frac{dt}{\phi(t)} \leq C.
\]
Indeed, the conditions (P1)-(P3) imply that for every $R\leq 1$ and $t\geq 0$ it holds $R^2\phi(t/R) \leq C\phi(t)$ (see Section 2). Therefore if $\phi$ satisfies the Osgood condition \eqref{osgood} the inequality  \eqref{osgood_harnack} quantifies  SMP. Moreover a naive example indicates  that  this estimate is  indeed sharp (see Section 3).  On the other hand,  in the homogeneous case   $\phi(t) = t $ \eqref{osgood_harnack}  reduces to  \eqref{cl_harnack}.

I need to assume that $\phi$ is of the form $\phi(t)= \eta(t)t$ and $\eta$ satisfies (P2), which implies that $\eta$ is  slowly increasing function. However,  I see no reason why Theorem \ref{weakHarnack} could not be true for any increasing and continuous function $\phi$. On the other hand, since  (P2) is only a condition of the asymptotic behaviour of $\phi$ at the infinity,  it plays no role in quantifying SMP. 

Theorem \ref{weakHarnack} also gives  an   upper bound for nonnegative solutions of \eqref{thePDE}. Namely, suppose that $\phi$ satisfies the Keller-Osserman type condition
\begin{equation}
\label{keller-osser}
\int_1^\infty \frac{dt}{\phi(t)}= \infty
\end{equation}
and $u\in C(B_{2R}(x_0))$ is as in Theorem  \ref{weakHarnack}. Then  $u$ is bounded  in $B_{R}(x_0)$ and the upper bound depends only on $\inf_{B_R(x_0)}u$. A simple example  (the construction is left to the reader)  shows that the condition \eqref{keller-osser} is sharp for such  upper bound estimate. In other words, if $\phi$ does not satisfy \eqref{keller-osser} we can construct a family of  nonnegative solutions $(u_k)$ of \eqref{thePDE} with $u_k(x_0)\leq 1$ which take arbitrary large values in $B_R(x_0)$. However,  even if \eqref{keller-osser} fails, the condition (P2) still implies   $\phi(t) \leq C t^{2-\eps}$. In this case  Theorem \ref{weakHarnack} gives a local upper bound estimate, i.e., if $u\in C(B_{2R}(x_0))$ is as in Theorem  \ref{weakHarnack} then there exists a radius $r \leq R$, which depends on the value $u(x_0)$, such that $u$ is bounded in  $B_{r}(x_0)$. We note that the condition \eqref{keller-osser} appears as the  sharp condition for the existence of entire solutions  \cite{FQS} and blow-up solutions \cite{AGQ} of 
\[
\Delta u =\phi(|Du|).
\]

The fact that the inequality \eqref{osgood_harnack}  depends on the  radius  is clear from the following    scaling argument.  If $u\in C(B_{2R}(x_0))$ satisfies \eqref{model1} and \eqref{model2} in $B_{2R}(x_0)$ then the rescaled function $u_R(x):= u(Rx)$ satisfies the same inequalities in $B_{2}(x_0)$ with  $\phi_R(t) = R^2\phi(t/R)$ instead of $\phi$. In the next section we show that   (P2)  implies that  $R^2\phi(\frac{t}{R}) \to 0$ locally uniformly in $[0,\infty)$ as $R \to 0$. Therefore  \eqref{osgood_harnack}  asymptotically converges to \eqref{cl_harnack} as the radius approaches to zero. This observation leads to   H\"older  estimate. This  has  already been obtained e.g. in \cite{Si}. 
\begin{corollary}
\label{holder}
Let   $u\in C(\Omega)$ be  a  viscosity supersolution of \eqref{model1} and a viscosity  subsolution of \eqref{model2} in $\Omega$. Then $u$ is locally  $\alpha$-H\"older continuous  with a uniform $\alpha \in (0,1)$, and  for every ball $B_{R_0}(x) \subset \subset  \Omega$ and $R\leq R_0 \leq 1$  we have
\[
\text{osc}_{B_{R}(x) }u \leq C \left( \frac{R}{R_0} \right)^{\alpha}
\]
where the constant $C$ depends on $\sup_{B_{R_0}(x) }|u|$.
\end{corollary}

As a further application of Theorem \ref{weakHarnack} we obtain the sharp Harnack type  of inequality for so called $p(x)$-harmonic functions. These are   local minimizers of the  energy
\[
\int_{\Omega} \frac{1}{p(x)} |Du|^{p(x)}\, dx,
\]
where $1 < p(x)< \infty$. This problem was first considered by Zhikov \cite{Zh} and have recently received a lot of attention. The regularity is rather well understood \cite{AM} and SMP was obtained in \cite{FZZ}, see also \cite{HHLT}. However, finding the sharp Harnack inequality which quantifies  SMP  has been an open problem  \cite{Alk}, \cite{HKLMP}, \cite{Wo}.   

We assume that  $p \in C^1(\R^n)$ and that there are numbers $1 <p^- \leq p^+ < \infty$ such that $p^- \leq p(x) \leq p^+$ for every $x \in \R^n$. It follows from \cite{JLP}  that  $p(x)$-harmonic functions are viscosity solutions of the Euler-Lagrange equation in nondivergence form
\begin{equation}
\label{p(x)-laplace}
- \Delta u - (p(x) -2) \Delta_{\infty}u - \log|Du| \langle Dp(x), Du \rangle = 0.
\end{equation}
 Here  $\Delta_{\infty}u= \langle D^2 u \frac{Du}{|Du|}, \frac{Du}{|Du|} \rangle$ denotes the infinity Laplace operator. It is not difficult  to show that if $u$ is a viscosity solution of \eqref{p(x)-laplace}, then it is a viscosity supersolution of \eqref{model1} and a viscosity subsolution of \eqref{model2} for $\lambda = \min\{ 1, p^- -1\}$, $\Lambda = \max \{ 1, p^+ -1 \}$ and 
\[
\phi(t) = C (|\log t|+ 1)t,
\]
where $C$ is the  $C^1$-norm of $p(\cdot)$  (see Lemma \ref{px_lemma}). Note that the above  function  satisfies the Osgood condition \eqref{osgood}. We have the following corollary. 
\begin{corollary}
\label{p(x)-harnack}
Let $p \in C^1(\R^n)$ be a function such that $1 < p^- \leq p(x) \leq p^+ < \infty$. Let  $u\in C(B_{2r}(x_0))$ be a nonnegative  $p(x)$-harmonic function  in  $B_{2R}(x_0)$ for  $R \leq 1$. Denote $m:= \inf_{B_R(x_0)}u$ and  $M:= \sup_{B_R(x_0)}u$. Then the following Harnack inequality holds
\[
\int_{m}^{M} \frac{ dt}{\left( R |\log t|+ 1 \right)t}  \leq C,
\]
where the constant $C$ depends on the dimension, on $L^{\infty}$-norm of $\nabla p$ and on  the numbers $p^-, p^+$.
\end{corollary}
 The previous estimate can be written more explicitly as
\[
\min\{M, M^{1+ CR} \} \leq C \max\{m, m ^{1+ C R}\}
\]	
by possibly enlarging the constant $C$. At the end of Section 5 we show that this result is optimal and  thus solves the problem of finding  the optimal generalization of Harnack inequality for $p(x)$-harmonic functions.  Note that we may relax the assumption of $p \in C^1(\R^n)$ to $p$ being  Lipschitz continuous by a standard approximation argument.

The paper is organized as follows. In the next section we recall some standard definitions and results. In  Section 3 we formally  prove Theorem \ref{weakHarnack} in dimension one. This easy proof  clarifies  where the estimate in the theorem comes from. Section 4 is devoted to the proof of the main result. The main difficulty is of course due to the fact that the equation is not scaling invariant. We refine the classical regularity theory by  introducing a new method which enables us to include the scaling of the equation in the regularity argument. The proof begins with a standard  decay estimate which states that for  nonnegative supersolution $u \in C(B_2)$   of \eqref{model1} in $B_2$  with   $\inf_{B_1}u \leq 1$  it holds 
\[
|\{x \in B_1 : u(x) \leq L \}|\geq \mu
\]
for universal $L>1$ and $\mu >0$. This is not difficult since the equation is uniformly elliptic. The result follows once we prove a decay estimate for the level sets $A_t= \{x \in B_1 : u(x) > t \}$. Due to the scaling of the equation  we observe that  one may use the previous decay estimate in \emph{small balls of radius} $r \leq \frac{t}{\phi(t)}$ which cover the $\frac{t}{\phi(t)}$-neighborhood of $\partial A_t$. This and the relative isoperimetric inequality then gives a lower bound for $| A_t \setminus A_{Lt}|$ (Lemma \ref{lemma_perus}). Finally a rather complicated  iteration argument   leads to estimate the decay of  $|A_t|$.  In the last section we show how  Corollary \ref{p(x)-harnack} follows from Theorem \ref{weakHarnack}.

At the end I would like to mention recent works \cite{AS}, \cite{IS}, \cite{Sa} which introduce new techniques for proving regularity estimates for nondivergence from elliptic equations. These papers study equations which are homogeneous but which fail to be uniformly elliptic. The key idea is to bypass the classical ABP estimate by directly touching the solution from below by paraboloids (or by other suitable functions).

\section{Preliminaries}

Let $X \in \mathbb{S}^{n \times n}$ be a symmetric $n$-by-$n$ matrix with eigenvalues $e_1, e_2, \dots, e_n$. The Pucci extremal operators $\Pu_{\lambda, \Lambda}^+$ and $\Pu_{\lambda, \Lambda}^-$ with ellipticity constants $0 < \lambda \leq \Lambda$ are defined by
\[
\Pu_{\lambda, \Lambda}^+(X):= -\lambda \sum_{e_i \geq 0} e_i - \Lambda \sum_{e_i<0} e_i  \qquad \text{and} \qquad \Pu_{\lambda, \Lambda}^-(X):= -\Lambda \sum_{e_i \geq 0} e_i - \lambda \sum_{e_i<0} e_i  .
\]
For elementary properties of the Pucci operators see \cite{CC}.  

We recall the definition of a viscosity supersolution of \eqref{model1} and a viscosity  subsolution  of \eqref{model2}. 
\begin{definition}
\label{visco_def}
A function $u: \Omega \to \R$ is a viscosity supersolution of \eqref{model1} in $\Omega$ if it is lower semicontinuous and the following holds: if $x_ 0 \in \Omega$ and $\vphi \in C^2(\Omega)$ is such that  $\vphi \leq u$ and $\vphi(x_0)= u(x_0)$  then 
\[
\Pu_{\lambda, \Lambda}^+(D^2\vphi(x_0)) \geq -\phi(|D\vphi(x_0)|).  
\]  
A function $u: \Omega \to \R$ is a viscosity subsolution of \eqref{model2} in $\Omega$ if it is upper semicontinuous and the following holds: if $x_ 0 \in \Omega$ and $\vphi \in C^2(\Omega)$ is such that $\vphi \geq u$ and $\vphi(x_0)= u(x_0)$ then 
\[
\Pu_{\lambda, \Lambda}^-(D^2\vphi(x_0)) \leq \phi(|D\vphi(x_0)|).  
\]  
\end{definition}

Finally we recall the definition of slowly increasing function. We will not need this property but it is important to connect the condition (P2) to the theory of  regularly varying functions,  see \cite{BGT}.  
\begin{definition}
\label{hitaasti_kasvava}
A nondecreasing and locally Lipschitz continuous  $f:[a, \infty) \to \R$  is \emph{slowly increasing function} if
\[
\lim_{t \to \infty} \frac{tf'(t)}{f(t)} \to 0.
\]
\end{definition}
The condition (P2) implies that  $\eta$ is slowly increasing.  In the next proposition we use this condition to study the asymptotic behaviour of  $\eta$. 
\begin{proposition}
\label{prop_hitaasti}
Let  $\eta$ satisfy the conditions (P1)-(P3). Then it holds:
\begin{itemize}
\item[(i)]  For every $c>0$ we have
\[
 \lim_{t \to \infty} \frac{\eta(ct)}{\eta(t)}=1.
\]
\item[(ii)] For every $\gamma>0$ we have
\[
\lim_{t \to \infty} \frac{\eta(t)}{t^{\gamma}} = 0.
\]
\item[(iii)] There is a constant $\Lambda_1$ such that for every $t>0$ it holds 
\[
\eta(\eta(t)t) \leq \Lambda_1 \eta(t).
\]
\item[(iv)] There is a constant $\Lambda_2$ such that for every $t>0$  and $0<r<s$ it holds 
\[
r \eta\left(t/r \right) \leq \Lambda_2 s \eta\left( t/s \right).
\]
\end{itemize}
\end{proposition}

\begin{proof}
The property (i) is called \emph{slowly varying property} and it  follows from the Representation Theorem  (\cite[Theorem 1.3.1]{BGT}), since $\eta$ is slowly increasing. The proof of (ii)  is elementary (see \cite[Proposition 1.3.6]{BGT}).

 It follows from the assumption (P2) that 
\[
\lim_{t \to \infty}\frac{\eta(\eta(t)t)}{\eta(t)}= 1,
\]
see  (\cite[Proposition 2.3.2 ]{BGT}) and  (\cite[Theorem 2.3.3]{BGT}). This implies  (iii) for every $t \geq 1$. Note   that $\phi(t) \geq t$ yields $\eta(t)\geq 1$. Therefore we obtain (iii) for every $t \in (0,1)$ by the monotonicity condition (P1).

We are left with (iv). By changing $t/s \mapsto t$ and $r/s \mapsto r$ the claim is equivalent to 
\[
r \eta(t/r) \leq \Lambda_2 \eta(t)
\] 
for every $t>0$ and $r < 1$. The   part (ii) yields $ \frac{\eta(1/r)}{1/r} \leq C$ for some constant $C$.  The condition (P3) then gives
\[
r \eta(t/r) \leq \Lambda_0 \frac{\eta(1/r)}{1/r} \eta(t) \leq \Lambda_2 \eta(t) .
\]
\end{proof}
Note that Proposition \ref{prop_hitaasti} (iv) implies $R^2\phi(t/R) \leq \Lambda_2 \phi(t)$ for every $R \leq 1$ and $t \geq 0$.

\section{The one-dimensional case}

In this section we formally   prove  Theorem \ref{weakHarnack} in dimension one. This is of course trivial compared to the general case, but the proof clarifies   why the result is true and where the estimate comes from. We also construct  a naive  example which indicates that Theorem \ref{weakHarnack} is sharp.

\begin{proof}[Proof of  Theorem \ref{weakHarnack} in dimension one] 
Let $R=1$, $x_0=0$ and let $u \in C_{loc}^{1,1}((-2,2))$  be a non-negative supersolution of  \eqref{model1} and subsolution of   \eqref{model2} in $(-2,2)$. In  one-dimension this simply means that 
\begin{equation}
\label{1Dmodel}
|u''(x)|\leq \lambda^{-1}\phi(|u'(x)|) \quad \text{a.e. }\, x \in (-2,2).
\end{equation}
Let us recall that the goal is to show
\[
\int_m^M \frac{dt}{\phi(t)+t} \leq C
\]
where
\[
m := \min_{-1 \leq x \leq 1} u(x)  \quad \text{and} \quad M := \max_{-1 \leq x \leq 1}u(x). 
\]

In order to simplify the proof let us restrict to the case when  $u$ is monotone, say nondecreasing.  Then we have $m = u(-1)$, $M = u(1)$ and  $ u(-2) \leq m$. By the mean value theorem there exist $\xi_1, \xi_2 \in (-2,1)$ such that 
\[
m-u(-2)= u'(\xi_1) \quad \text{and} \quad M- u(-2)= 3u'(\xi_2) .
\]
Note that if $M/6 \leq m$ then the classical Harnack's inequality  holds and the claim follows. Thus we treat the case $M/6 \geq m$. We denote $a:= u(-2)$ and use the  above estimate to deduce  
\begin{equation}
\label{1D_ABP}
\big[m-a, \frac{M-a}{3}\big] \subset u'((-2,1)).
\end{equation}
Note that  $M/6 \geq m$ implies that  the lenght of the interval $[m-a, M/3-a/3]$ is bigger that $M/6$. We use the monotonicity of $\phi$,  \eqref{1Dmodel} and \eqref{1D_ABP} to obtain
\[
\begin{split}
\int_{m}^{M/3} \frac{dt}{\phi(t)} &\leq \int_{m-a}^{\frac{M-a}{3}} \frac{dt}{\phi(t)} \leq \int_{u'((-2,1))} \frac{dt}{\phi(t)} \leq \int_{-2}^1 \frac{|u''(x)|}{\phi(u'(x))}dx \leq 3\lambda^{-1}.
\end{split}
\]  
Since $M/3-m \geq M/6$ , the monotonicity of $\phi$ and the above estimate yield
\[
\int_{kM/6}^{(k+1) M/6} \frac{dt}{\phi(t)} \leq 3\lambda^{-1}
\]
for $k = 1, 2,3,4,5$. This implies the result in the case $R=1$. The general case $R\leq 1$ follows   by a simple scaling argument.  
\end{proof} 

Let us construct an example which indicates  that the result  is sharp. To that aim we  assume that $\phi$   is $C^1$-regular, it satisfies the Osgood condition \eqref{osgood} and  (P1)-(P3), and  that 
\begin{equation}
\label{toisin_pain}
\eta(t)\leq C \eta(\eta(t)t) \quad  \text{for } \,  t >0. 
\end{equation}
By the inverse function theorem,  for every $k \in \N$,   there is a positive increasing function $u_k:(-2,2) \to (0,\infty)$ such that $u_k(0)= \frac{1}{k}$ and
\[
x = \int_{\frac{1}{k}}^{u_k(x)} \frac{dt}{\phi(t)} \qquad x \in (-2,2).
\]
By differentiating this we obtain 
\begin{equation}\label{lelu_ody}
\frac{u_k'}{\phi(u_k)}  = 1 \qquad \text{and} \qquad u_k''= \phi'(u_k) u_k'.
\end{equation}
Recall that $\phi(t)= \eta(t)t$. Since $\eta$ is slowly increasing and  \eqref{toisin_pain}  holds it is easy to see that \eqref{lelu_ody} gives
\[
 \phi'(u_k) \leq C \eta(u_k) \leq  C \eta(\eta(u_k)u_k) = C  \eta(\phi(u_k)) =  C  \eta(u_k'). 
\]
Therefore \eqref{lelu_ody} implies that $u_k$ satisfies
\[
|u_k''(x)| \leq C \phi(u_k'(x)).
\]
Since $u_k$ is nondecreasing we have $m = \inf_{-1<x<1 }u_k(x) =  u(-1)$ and $M= \sup_{-1<x<1 }u_k(x) = u_k(1)$, and therefore  
\[
2 = \int_{\frac{1}{k}}^{u_k(1)} \frac{dt}{\phi(t)} - \int_{\frac{1}{k}}^{u_k(-1)} \frac{dt}{\phi(t)} = \int_{m}^{M} \frac{dt}{\phi(t)}.
\] 
This estimate  is  optimal. In particular, it is not true that the ratio
\[
\frac{u_k(1)}{u_k(-1)}
\]
is uniformly bounded for $k \in \N$.

\section{Proof of the  Harnack inequality}

In this section we prove Theorem \ref{weakHarnack}. We begin by  following  the proof of the Krylov-Safonov theorem found by Caffarelli (\cite{C1}, \cite{CC}) and obtain a decay estimate on small scales (Lemma \ref{decay_est}).    In the first lemma we construct a barrier function. The proof can be found in Appendix \ref{appendix}.
\begin{lemma}
\label{barrior}
There is a smooth function $\vphi$ in $\R^n$ and universal constants $L_1$ and $r_0 \in (0,1]$ such that 
\begin{enumerate}
\item[(i)] $\vphi = 0$ on  $\partial  B_{2r_0}$,
\item[(ii)] $\vphi \leq -2 $ in  $B_{r_0}$,
\item[(iii)] $\Pu_{\lambda, \Lambda}^-(D^2 \vphi) \geq \phi(|D\vphi|) -  C\xi$ in  $B_{2r_0}$,  
\end{enumerate} 
where $0\leq \xi \leq 1$ is a continuous function such that $\text{supp} \, \xi  \subset B_{\frac{r_0}{2}}$. Moreover, $\vphi \geq -L_1$ in $B_{2r_0}$ and $|D \vphi| \geq L_1^{-1}$ in $B_{2r_0} \setminus B_{\frac{r_0}{2}}$.
\end{lemma}

 In the second lemma we use the ABP estimate and  the previous barrier function to obtain the decay estimate. The proof is again in Appendix \ref{appendix}.

\begin{lemma}
\label{decay_est}
Let  $u \in C(B_{2})$  be a nonnegative supersolution of \eqref{model1} in $B_{2}$. If $\inf_{B_{1}} u \leq 1$ then it holds
\[
|\{x\in B_1  \mid   u(x) \leq L_1\}| > \mu,
\]
for universal constants $L_1>1$ and $\mu>0$. 
\end{lemma} 

Let us next reformulate the statement of Theorem \ref{weakHarnack}.

\begin{remark} \label{triviaali_ska}
Suppose that $v \in C(B_{2R}(x_0))$ satisfies the assumptions of Theorem \ref{weakHarnack}. Then the function $v_R \in C(B_{2})$
\[
v_R(x) := \frac{1}{R} v(Rx+x_0)
\]
is a supersolution of 
\begin{equation}
\label{model1.1}
\Pu_{\lambda, \Lambda}^+(D^2 u)  \geq -R\phi(|Du|)  
\end{equation}
and a subsolution of 
 \begin{equation}
\label{model2.2}
\Pu_{\lambda, \Lambda}^-(D^2 u)  \leq R\phi(|Du|). 
\end{equation}
Therefore in order to prove Theorem \ref{weakHarnack} we need to show that   a nonnegative function $u \in C(B_{2})$ which is a viscosity supersolution of  \eqref{model1.1}  and a viscosity subsolution of \eqref{model2.2} in $B_2$ satisfies 
\[
\int_m^M \frac{dt}{R\phi(t)+ t} \leq C,
\]
where $m = \inf_{B_1}u$ and $M= \sup_{B_1}u$. 
\end{remark}

We now come  to the point where the proof of Theorem \ref{weakHarnack} truly differs from the classical  proof  of  Krylov-Safonov theorem. Since the equation \eqref{model1} is not scaling invariant we can not simply iterate Lemma \ref{decay_est}.  We overcome this problem by the following scaling argument. If $u$ is a supersolution of \eqref{model1} and $A>0$ is a given number,  we can  find   $r_A>0$ such that the rescaled  function
\[
\tilde{u}(x) = \frac{u(r x)}{A}
\] 
is again a supersolution of  \eqref{model1} for all $r \leq r_A$. This will be done in the next lemma, which  we will  use frequently. In this lemma we  need the assumption (P2) on  $\phi$.

\begin{lemma}
\label{skaalaus}
Let  $u \in C(B_{2})$  be a  supersolution of \eqref{model1.1} in $B_{2}$ for $R\leq 1$. There exists a universal constant $L_2$ such that if   $A \in (0,\infty)$ then for every $r \leq r_A$ for
\[
r_A = \frac{1}{L_2}\frac{A}{R\phi(A) + A} = \frac{1}{L_2 (R\eta(A)+ 1)}
\]
the rescaled function
\[
\tilde{u}(x) :=  \frac{u(rx)}{A}
\]
is a supersolution of \eqref{model1} in its domain, i.e., $\tilde{u}$ is a supersolution of  \eqref{model1.1} for $R=1$.
\end{lemma}

\begin{proof}
We have
\[
D \tilde{u}(x) = \frac{r}{A}Du(rx) \quad \text{and} \quad D^2 \tilde{u}(x) =  \frac{r^2}{A}Du(rx).
\]
Hence, $\tilde{u}$ is a supersolution of 
\[
\Pu_{\lambda, \Lambda}^+( D^2 \tilde{u}(x))  \geq  - \frac{r^2R}{A} \phi\left( \frac{A}{r}|D \tilde{u}(x) |\right) 
\]
in its domain. Therefore the claim follows once we show
\[
 \frac{r^2R}{A} \phi\left(At/r \right) \leq \phi(t) \quad \text{for every }\, t>0.
\]
Since  $\phi(t)= \eta(t)t$ this  is equivalent to 
\begin{equation}
\label{re_skaalaus}
rR \, \eta\left(At/r \right) \leq  \eta(t).
\end{equation}
Recall that it holds $\eta(t)\geq 1$.

Proposition \ref{prop_hitaasti} (iii)  and the monotonicity condition (P1) imply 
\[
 \frac{\eta((R\eta(A)+1)A) }{\eta(A)}\leq C_0
\]
for a universal constant $C_0$. We use  Proposition \ref{prop_hitaasti} (iv) and the conditions (P1) and (P3) to conclude that for some constant $C_0$ it holds
\[
\begin{split}
rR \, \eta\left(\frac{At}{r} \right)  &\leq C_0 r_A R \, \eta\left(\frac{At}{r_A} \right) \\
&\leq \frac{C_0}{L_2 \eta(A)} \eta( L_2 (R\eta(A)+1) A t) \\
&\leq C_0^2 \frac{\eta(L_2) }{L_2} \frac{\eta((R\eta(A)+1)A) }{\eta(A)}  \eta(t)\\
&\leq  C_0^3 \frac{\eta(L_2) }{L_2}   \eta( t). 
\end{split}
\]
By Proposition \ref{prop_hitaasti} (ii) we may choose $L_2$ such that 
\[
  \frac{\eta(L_2) }{L_2}    \leq \frac{1}{C_0^3}
\]
and  \eqref{re_skaalaus} follows. 
\end{proof}

We denote the open $\delta$-neighbourhood of a set $S \subset \R^n$ by
\begin{equation} 
\label{delta_neighbourhood}
\mathcal{I}_{\delta}(S) = \{ x \in \R^n \mid \dist(x,S) < \delta\}.
\end{equation} 
We need the following corollary of the relative isoperimetric inequality.
\begin{lemma}
\label{isop_ineq}
 There is a dimensional constant $c_n>0$ such that for every set  $E \subset B_R$ and for every $\delta \in (0,1)$ it holds
\[
\big| \mathcal{I}_{\delta}(\partial E \cap B_R)\cap E \big|  \geq c_n\min \Bigl\{  \delta  |B_R \setminus E|^{\frac{n-1}{n}},  \delta  |E |^{\frac{n-1}{n}} , |E|\Bigl\}.
\]
\end{lemma}

\begin{proof}
We may assume  that $|E| \leq  |B_R \setminus E|$, for  in  the case $|E| \geq  |B_R \setminus E|$ the argument is similar. Let us  denote
\[
E_{s}:= \{ x \in E \mid d_{\partial E}(x) > s\}
\]
where $d_{\partial E}(x)= \dist(x, \partial E \cap B_R)$. Notice that $E_{t} \subset E_s$ for every $s <t$ and $|E_{\delta}| \leq \frac{1}{2}|B_R|$. If $|E_{\delta}| \leq  \frac{1}{2}|E|$, we have
\[
| \mathcal{I}_{\delta}(\partial E \cap B_R)\cap E| = |E \setminus E_{\delta}| =|E| -  |E_{\delta}|  \geq  \frac{1}{2}|E|
\]
and the claim follows. Let us then treat the case   $|E_{\delta}| \geq  \frac{1}{2}|E|$.  Note that then $|E_{s}| \geq  \frac{1}{2}|E|$ for every $s \in (0, \delta)$. 

Lipschitz continuity of $d_{\partial E}$ implies that for almost  every $s \in (0, \delta)$  the  set $E_s$ has finite perimeter \cite[Theorem 3.40]{AFP}. By the relative  isoperimetric inequality \cite[Remark 3.50] {AFP} we have
\begin{equation}
\label{rel_isop_ineq}
P(E_s, B_R) \geq \tilde{c}_n \min \{|B_R \setminus E_s|^{\frac{n-1}{n}} , |E_s|^{\frac{n-1}{n}} \} \geq c_n |E|^{\frac{n-1}{n}}
\end{equation}
for some dimensional constants $\tilde{c}_n$ and $c_n$, where the last inequality follows from  $\frac{1}{2}|E| \leq |E_s| \leq  |B_R \setminus E_s|$. Here $P(E_s, B_R)$ denotes the perimeter of $E_s$ in the ball $B_R$. Since $|\nabla d_{ E}(x)|= 1$ for almost every $x \in   E$, the coarea formula in $BV$  \cite[Theorem 3.40]{AFP} and \eqref{rel_isop_ineq} yield
\[
\begin{split}
|\mathcal{I}_{\delta}(\partial E \cap B_R) \cap E| = |E \setminus E_{\delta}| &= \int_{E \setminus E_{\delta}}|\nabla d_{E}(x)|\, dx\\
&\geq \int_0^{\delta} P(E_s, B_R)\, d s \geq  c_n \delta |E|^{\frac{n-1}{n}}.
\end{split}
\] 
\end{proof}

In the next lemma we study the decay of the level sets
\[
\{ x \in B_{\frac{5}{3}} \mid u(x) > t\}.
\]
From now on we assume that $u$  is a positive  supersolution of \eqref{model1.1} for $R\leq 1$  in $B_2$ and  denote $m= \inf_{B_1}u>0$. We study the level sets
\begin{equation}
\label{def.A_k}
A_k := \{ x \in B_{\frac{5}{3}} \mid u(x) > L^km\}
\end{equation}
where $L$ is a uniform constant which  will be chosen later.

As we mentioned in the introduction the proof is based on an observation that for a fixed $k$, due to a scaling argument and Lemma \ref{skaalaus}, we may use a rescaled version of Lemma \ref{decay_est}  in a $a_k$-neighborhood of $\partial A_k$  in $A_k$ where
\begin{equation}
\label{def.a_k}
a_k := \frac{L^{k-1}m}{R \phi(L^km) + L^km} = \frac{1}{ L} \cdot \frac{1}{R\eta(L^km)+1}.
\end{equation}
A standard covering argument then implies that  part  of the  $a_k$-neighborhood of $\partial A_k$ in $A_k$ does not belong to the next level set $A_{k+1}$. We then  use Lemma \ref{isop_ineq} to estimate the size of the set $A_k \setminus A_{k+1}$. 
\begin{lemma}
\label{lemma_perus}
Let  $u$  be a positive  supersolution of \eqref{model1.1} for $R\leq 1$ in $B_2$ and let $L= \max\{L_1, L_2,6 \}$,  where $L_1, L_2$ are the  constants  from Lemma \ref{decay_est} and Lemma \ref{skaalaus}. Let us  denote $m= \inf_{B_1}u$ and suppose that  the set $A_k$ is defined in  \eqref{def.A_k} and the number $a_k$ in \eqref{def.a_k}. Then it holds
\[
|A_k \setminus A_{k+1}|  \geq c_0  \min \{ a_k|B_{\frac{5}{3}} \setminus A_k|^{\frac{n-1}{n}}, a_k|A_k|^{\frac{n-1}{n}} , |A_k|\},
\]
for a uniform constant $c_0 \in (0,1)$.
\end{lemma}

\begin{proof}
Let us  fix $k \geq 1$. Note that 
\[
A_k \setminus A_{k+1} = \{ x \in B_{\frac{5}{3}} : L^km< u \leq  L^{k+1}m \}.
\]  
We denote the open $a_k$-neighborhood of $\partial  A_k$  in $A_k$ by 
\[
D_k := \mathcal{I}_{a_k}(\partial  A_k \cap B_{\frac{5}{3}}) \cap A_k = \{ x \in A_k : \dist(x, \partial A_k \cap  B_{\frac{5}{3}})< a_k\}.
\]
Note that if $A_k \cap B_{\frac{5}{3}} = \emptyset$ the claim is trivially true. Let  $\mathcal{B}$ be collection of all balls $B_{r}(x) \subset  A_k$ with
\[
r = \dist(x, \partial A_k \cap B_{\frac{5}{3}}) \leq a_k \leq \frac{1}{6}. 
\] 
Note that $ B_{2r}(x) \subset B_2$. It is then easy to see that the balls  $ B_{5r}(x) $ , with $ B_{r}(x) \in \mathcal{B}$,  cover $D_k$. By Vitali's covering theorem we may choose a countable subcollection from $\mathcal{B}$, say $ B_{r_i}(x_i)$, which are disjoint and the balls $B_{5^2r_i}(x_i)$ still cover $D_k$. Lemma \ref{isop_ineq} implies
\begin{equation}
\label{using_isop}
\begin{split}
5^{2n}\sum_i |B_{r_i}(x_i)| &=  \sum_i |B_{5^2r_i}(x_i)| \geq |D_k| \\
&\geq c_n\min \Bigl\{  a_k  |B_{\frac{5}{3}} \setminus A_k|^{\frac{n-1}{n}},  a_k  |A_k|^{\frac{n-1}{n}} , |A_k|\Bigl\}.
\end{split}
\end{equation}

Let us fix a ball $B_{r_i}(x_i)$ which belongs to the Vitali cover and rescale $u$ by
\[
\tilde{u}(x):= \frac{u(r_ix + x_i)}{L^km}.
\]
Then $\tilde{u}$ is nonnegative in $B_2$ and $\inf_{B_1} \tilde{u} \leq 1$, which follows from  $\partial B_{r_i}(x_i) \cap \partial A_k \neq \emptyset$. Moreover, since 
\[
r_i \leq a_k = \frac{L^{k-1}m}{R\phi(L^km)+ L^km} \leq  \frac{1}{L_2}\frac{L^{k}m}{R\phi(L^km) + L^km} 
\]
 we deduce from Lemma \ref{skaalaus} that  $\tilde{u}$  is a supersolution of \eqref{model1} in $B_2$. We may thus apply Lemma \ref{decay_est} to conclude that
\begin{equation}
\label{using_decay}
\begin{split}
\frac{|B_{r_i}(x_i) \setminus A_{k+1}|}{|B_{r_i}| } &=  \frac{|\{ x \in B_{r_i}(x_i) : u(x) \leq L^{k+1}m\}|  }{|B_{r_i}| }\\
&=  |\{ y \in B_{1} : \tilde{u}(y) \leq L \}| \\
&\geq   |\{ y \in B_{1} : \tilde{u}(y) \leq L_1 \}| \geq \mu.
\end{split}
\end{equation}

Since $B_{r_i}(x_i) \subset A_k$ are disjoint, we obtain from \eqref{using_isop} and  \eqref{using_decay} that 
\[
\begin{split}
|A_k \setminus A_{k+1}|   &\geq \sum_i |B_{r_i}(x_i) \setminus A_{k+1}| \\
&\geq \mu \sum_i |B_{r_i}(x_i) |\\
&\geq 5^{-2n} c_n \mu \min \Bigl\{  a_k  |B_{\frac{5}{3}} \setminus A_k|^{\frac{n-1}{n}},  a_k  |A_k|^{\frac{n-1}{n}} , |A_k|\Bigl\}.
\end{split}
\]
\end{proof}

We continue by  iterating the estimate from  Lemma \ref{lemma_perus}. Due to the relative isoperimetric inequality the iteration behaves differently depending on the size of $A_k$.
\begin{lemma}
\label{lemma_alku}
Let  the function $u$ be as in Lemma \eqref{lemma_perus} and suppose the sets $A_k \subset B_{\frac{5}{3}}$  and the numbers $a_k \in (0,1)$ are given by \eqref{def.A_k} and \eqref{def.a_k}.
\begin{itemize}
\item[(a)]  If there is $\delta \in (0,1)$ such that  $|A_j| \geq \delta$  for every $j = 0,1, \dots k$ where $k\geq 1$, then  it holds
\begin{equation} \label{decay_alku} 
|A_k|\leq |B_{\frac{5}{3}}| - c\left( \sum_{j=0}^{k-1}a_j \right)^n 
\end{equation}
for a  constant $c>0$ which depends  on $\delta$.
\item[(b)] If there is $k_0 \in \N$ such that $a_j^n \leq |A_j| \leq  \frac{1}{2n^n}$  for every $j = k_0,k_0 + 1, k_0 +2, \dots, k$ for $k > k_0$, then  it holds
\begin{equation} \label{decay_loppu} 
|A_{ k}| \leq  \frac{1}{n^n}\left(1- c \sum_{j=k_0}^{k-1}a_j \right)^n
\end{equation}
for a  universal constant $c>0$. 
\end{itemize}
\end{lemma}

\begin{proof}
Let us first prove (a). Since $|A_k|\geq  \delta$, we may use  Lemma \ref{lemma_perus} to deduce
\begin{equation}
\label{isoper_alku}
|A_k \setminus A_{k+1}|  \geq \tilde{c} a_k  \left(|B_{\frac{5}{3}}|-|A_k|\right)^{\frac{n-1}{n}}
\end{equation}
for $\tilde{c} = \frac{c_0 \delta}{|B_2|}$, where $c_0$ is the constant from Lemma \ref{lemma_perus}. Moreover, we may assume that $\tilde{c}  \leq \mu$, where $\mu$ is from Lemma \ref{decay_est}, by possibly decreasing $\tilde{c}$.

We  make  a few observations on the sequence   $(a_k)$ defined in \eqref{def.a_k}. First of all, since $\phi$ is increasing we have
\[
\frac{a_{k}}{a_{k-1}} = L \frac{R\phi(L^{k-1} m)+ L^{k-1} m}{R\phi(L^{k} m)+ L^{k} m} \leq L
\]
for every $k = 1,2, 3, \dots$. In particular, it holds 
\[
a_k \leq L \sum_{j=0}^{k-1}a_j
\]
for every $k = 1,2,3, \dots$. Therefore we find $N>0$ such that 
\begin{equation}
\label{helppo_monotonisuus}
\begin{split}
\left(\sum_{j=0}^{k-1}a_j   + a_k \right)^n &= \sum_{i=0}^n  \binom{n}{i}  \left(\sum_{j=0}^{k-1}a_j \right)^{n-i}a_k^i \\
&\leq \left(\sum_{j=0}^{k-1}a_j  \right)^n + N a_k \left(\sum_{j=0}^{k-1}a_j  \right)^{n-1}.
\end{split}
\end{equation}

Let us prove the claim by induction for   $c_1 = \left( \frac{\tilde{c}}{N} \right)^n$, where $\tilde{c}$ is the constant from \eqref{isoper_alku}.  We begin by observing that the only information from the set  $A_0 = \{x \in B_{\frac{5}{3}} : u(x)>m \}$ is that there is a point $\hat{x} \in \bar{B}_1$ such that $u(\hat{x})= m$.  Let us  choose $x_0 \in B_1$ such that $\hat{x} \in \bar{B}_{a_0}(x_0) \subset \bar{B}_1$. Since $a_0 \leq \frac{1}{6}$ we have $B_{2a_0}(x_0) \subset  B_{\frac{5}{3}}$. We argue as in the previous lemma and rescale $u$ by
\[
\tilde{u}(x)= \frac{u(a_0x + x_0)}{m}.
\] 
Then $\tilde{u}$ is nonnegative in $B_2$, $\inf_{B_1} \tilde{u} \leq 1$ and by Lemma \ref{skaalaus} it  is a supersolution of \eqref{model1} in $B_2$. We may thus apply Lemma \ref{decay_est} to conclude that
\[
\begin{split}
\frac{|B_{\frac{5}{3}}  \setminus A_{1}|}{|B_{2a_0} | } &\geq  \frac{|\{ x \in B_{2a_0}(x_0) : u(x) \leq L m\}|  }{|B_{2a_0}| }\\
&=  |\{ y \in B_{2} : \tilde{u}(y) \leq L \}| \\
&\geq \mu.
\end{split}
\]
Therefore we have
\[
\begin{split}
|A_1| &\leq |B_{\frac{5}{3}} | -\mu |B_{2a_0} |  \\
&\leq |B_{\frac{5}{3}} | - \left( \frac{\tilde{c}}{N} \right)^n  a_0^n
\end{split}
\]
where the last inequality follows from  $\tilde{c} \leq \mu$. Hence, the claim  holds for $k=1$.

We assume that the claim holds for $k >1$, i.e., 
\begin{equation} \label{induktio_alku} 
|A_k|\leq |B_{\frac{5}{3}} | - \left( \frac{\tilde{c}}{N} \right)^n \left( \sum_{j=0}^{k-1}a_j \right)^n. 
\end{equation}
We have by  \eqref{isoper_alku},  \eqref{helppo_monotonisuus}  and \eqref{induktio_alku}  that
\[
\begin{split}
|A_{k+1}|  &\leq  |A_k| - \tilde{c} a_k \left( |B_{\frac{5}{3}}|-  | A_k|\right)^{\frac{n-1}{n}} \\
&\leq |B_{\frac{5}{3}}| - \left( \frac{\tilde{c}}{N} \right)^n \left(  \left(  \sum_{j=0}^{k-1}a_j \right)^n + Na_k   \left(  \sum_{j=0}^{k-1}a_j \right)^{n-1}\right)\\
&\leq |B_{\frac{5}{3}}| - \left( \frac{\tilde{c}}{N} \right)^n \left(  \sum_{j=0}^{k}a_j \right)^n,
\end{split}
\]
which proves \eqref{decay_alku}.

We now prove (b). In this case  Lemma \ref{lemma_perus} implies  
\[
|A_k \setminus A_{k+1}|  \geq  c_0 a_k|A_k|^{\frac{n-1}{n}}
\]
for all $k \geq k_1$.  In other words
\begin{equation} \label{isoper_loppu} 
|A_{k+1}| \leq |A_k| - c_0 a_k|A_k|^{\frac{n-1}{n}}.
\end{equation}
By possibly decreasing $c_0$ we may assume that $c_0 \leq 1 - 2^{-1/n}$. Therefore it follows from $a_{k_1}\leq \frac{1}{6}$ that
\[
|A_{k_1+1}| \leq \frac{1}{2n^n} \leq \frac{1}{n^n} (1 -  c_0  a_{k_1} )^{n}.
\]
Hence the claim holds for $k=k_1+1$.

Let us assume that the claim holds for $k > k_1+1$, i.e., 
\begin{equation} 
\label{induktio-oletus}
|A_{ k}| \leq  \frac{1}{n^n}\left(1- c_0 \sum_{j=k_1}^{k-1}a_j \right)^n.
\end{equation}
Notice  that the assumption $|A_{ k}|\geq a_{k}^n$ implies 
\[
c_0 a_{k}\leq |A_{ k}|^{\frac{1}{n}} \leq  \frac{1}{n}\left(1-  c_0 \sum_{j= k_1}^{ k-1}a_j \right).
\]
Next we remark that if  there are positive numbers $a$ and $b$ such that  $b\leq\frac{1}{n} a$, then it holds
\[
(a- b)^n =  \sum_{i=0}^n \binom{n}{i} a^{n-i} (-b)^i \geq a^n - na^{n-1}b.
\]
The two previous inequalities  yield
\begin{equation}
\label{binomi}
\left((1- c_0 \sum_{j= k_1}^{k-1}a_j) - c_0 a_k \right)^n \geq  \left(1- c_0 \sum_{j=k_1}^{k- 1}a_j \right)^n -  n c_0 a_k \left(1-  c_0 \sum_{j=k_1}^{k-1}a_j \right)^{n-1}.
\end{equation}

Notice that the function  $t \mapsto t - c_0 a_k t^{\frac{n-1}{n}}$ is increasing in $[a_k^n, 1]$.  Since $|A_k| \geq a_k^n$ we have by \eqref{isoper_loppu}  and \eqref{induktio-oletus} that
\[
\begin{split}
|A_{k+1}| &\leq  |A_k| -  c_0a_k|A_k|^{\frac{n-1}{n}}\\
 &\leq  \frac{1}{n^n} \left( \left(1- c_0 \sum_{j=k_1}^{k-1}a_j \right)^n -    n  c_0 a_k \left(1-  c_0 \sum_{j=k_1}^{k-1}a_j \right)^{n-1} \right) \\
&\leq  \frac{1}{n^n}  \left(1- c_0 \sum_{j=k_1}^{k}a_j \right)^n,
\end{split}
\]
where the last inequality follows from \eqref{binomi}. 
\end{proof}

The next  lemma asserts that when the level sets $A_k$ are very small they finally decay as in the homogeneous case.  Roughly speaking  this means that the asymptotic behaviour of an unbounded supersolution of \eqref{model1} is completely determined by the second order operator, not the lower order drift term.  

\begin{lemma}
\label{lemma_huippu2}
Let  $u$ be  a positive supersolution of \eqref{model1.1} for $R\leq 1$ in $B_2$, and suppose  the sets $A_k \subset B_{\frac{5}{3}}$ and the numbers $a_k \in (0,1)$  are given by \eqref{def.A_k} and  \eqref{def.a_k}. Let the constant $c_0$ be as in Lemma \ref{lemma_perus} and denote $m = \inf_{B_1} u$. There is a universal constant $C$ such that either $\sum_{j=0}^{\infty} a_j \leq C$  or there is  an index $k_1 \in \N$ such that 
\begin{equation}
\label{summa-raja}
\sum_{j=0}^{k_1} a_j \leq C \qquad \text{and} \qquad |A_{k}| \leq (1-c_0) ^{k-k_1}a_{k_1}^n, \quad k \geq k_1.
\end{equation}
In the latter case we have 
\begin{equation}
\label{epsilon_muoto}
|\{ x \in B_{\frac{5}{3}} : u(x)> t L^{k_1}m\}|\leq c_{\eps} t^{-\eps} a_{k_1}^n \qquad \text{for }\,  t \geq 1
\end{equation}
where $c_{\eps}$ and $\eps>0$ are  universal constants.
\end{lemma}

Before the proof I would like to point out that there is no bound for the index $k_1$. The point  is that only the sum \eqref{summa-raja} is uniformly bounded.

\begin{proof}
Let us begin with a few preparations. Let $c_0$ and $L$ be the constants from Lemma \ref{lemma_perus}.  By Proposition \ref{prop_hitaasti} (i) there is a number $t_L\geq 1$ such that 
\begin{equation}
\label{kaukana}
\sup_{t \geq t_L}\frac{\eta(Lt) + R^{-1}}{\eta(t) + R^{-1}} \leq \sup_{t \geq t_L}\frac{\eta(Lt) }{\eta(t)} \leq \left(\frac{1}{(1-c_0)}\right)^{1/n}.
\end{equation}
Moreover,  Proposition \ref{prop_hitaasti} (ii) implies that   
\begin{equation}
\label{kaukana2}
\delta_0 := \inf_{j \in \N}  \left(  \frac{(1-c_0)^{-j}}{(R\eta(L^jt_L)+1)^n}\right)  >0.
\end{equation}

We begin by using Lemma \ref{lemma_alku} (a) for $\delta= \min \{\frac{\delta_0}{L^n},  \frac{1}{2n^n} \}$, where $\delta_0$ is defined in \eqref{kaukana2}. We conclude that there is  $C_0$, which depends only on $\delta$, such that either $\sum_{j=0}^{\infty} a_j \leq C_0$ or there is   $k_0 \in \N$ such that 
\[
\sum_{j=0}^{k_0-1} a_j \leq C_0 \qquad \text{and} \qquad |A_{k}| \leq   \delta
\]
 for every $k \geq k_0$. In the latter case Lemma \ref{lemma_alku} (b)   implies  that there is $C_1$ such that either $\sum_{j=k_0}^{\infty} a_j \leq C_1$ or there   is $k_1 > k_0$ such that 
 \[
\sum_{j=k_0}^{k_1-1} a_j \leq C_1  \qquad \text{and} \qquad|A_{k_1}| \leq  a_{k_1}^n.
\]
For every $k \geq k_1$ Lemma \ref{lemma_perus} gives
\begin{equation} \label{isoper_huippu2}
|A_k \setminus A_{k+1}| \leq c_0  \min \{  a_k|A_k|^{\frac{n-1}{n}} , |A_k|\}.
\end{equation}
In particular, we have 
\begin{equation} \label{huippu3}
|A_{k_1+1}| \leq  (1-c_0)|A_{k_1}|.
\end{equation}
 We  divide the rest of the proof into two cases.

 \textbf{Case 1:  $L^{k_1}m \geq t_L$}. 

In this case it follows from \eqref{kaukana} that 
\begin{equation} \label{kasvuehto2}
\frac{a_{k}}{a_{k+1}}= \frac{\eta(L^{k +1}m) + R^{-1}}{\eta(L^{k}m)+ R^{-1}} \leq \left(\frac{1}{(1-c_0)}\right)^{1/n}
\end{equation}
for every $k\geq k_1$. I claim that it holds 
\begin{equation} \label{huippu_iteraatio2}
|A_{k+1}| \leq  (1-c_0)|A_{k}|  \qquad \text{and} \qquad |A_{k}| \leq a_{k}^n
\end{equation}
for every $k \geq k_1$, which yields the claim. 

Indeed, \eqref{huippu3} implies that \eqref{huippu_iteraatio2} holds for $k=k_1$.  Assume that \eqref{huippu_iteraatio2} holds for $k >k_1$.  It follows from the induction assumption and \eqref{kasvuehto2} that
\[
|A_{k+1}| \leq  (1-c_0)|A_{k}| \leq (1-c_0)a_{k}^n \leq a_{k+1}^n.
\]
Hence, \eqref{isoper_huippu2} yields 
\[
|A_{k+2}| \leq  (1-c_0)|A_{k+1}|
\]
and  \eqref{huippu_iteraatio2} follows.

\bigskip
 \textbf{Case 2:  $L^{k_1}m \leq t_L$}. 
Let us prove  that also  in this case we have  
\begin{equation} \label{huippu_iteraatio3}
|A_{k}| \leq  (1-c_0)^{k-k_1}|A_{k_1}| 
\end{equation}
for $k > k_1$. Again \eqref{huippu3} implies that the claim holds for $k = k_1+1$.  Moreover, let us  recall that we have  
\[
|A_{k_1}| \leq  \delta \leq \frac{\delta_0}{L^n}
\]
where $\delta_0$ is given by \eqref{kaukana2}.

Assume \eqref{huippu_iteraatio3} is true for $k > k_1+1$. Assume first that    $L^{k}m <1$. Since $\eta$ is nonincreasing in $(0,1)$ we have
\[
a_{k}= \frac{1}{L}\frac{1}{R\eta(L^{k}m)+ 1} \geq  \frac{1}{L}\frac{1}{R\eta(L^{k_1}m)+ 1} = a_{k_1} \geq |A_{k_1}|^{1/n} \geq |A_{k}|^{1/n}. 
\]
Therefore \eqref{isoper_huippu2} gives 
\[
|A_{k+1}| \leq  (1-c_0)|A_{k}| \leq  (1-c_0)^{k+1-k_1}|A_{k_1}|. 
\]
On the other hand, if $L^{k}m \geq 1$ the assumption $L^{k_1}m\leq t_L$ yields
\[
\eta(L^{k}m) \leq \eta(L^{k-k_1}t_L)
\]
since $\eta$ is nondecreasing in $[1, \infty)$. We use  \eqref{kaukana2} to deduce that 
\[
\begin{split}
a_{k}^n &\geq \frac{1}{L^n}\frac{1}{(R\eta(L^{k-k_1}t_L) +1)^n }  \\
&\geq (1-c_0)^{k-k_1} \frac{\delta_0}{L^n}\\
&\geq (1-c_0)^{k-k_1}|A_{k_1}| \\
&\geq |A_{k}|,
\end{split}
\]
where the last inequality follows from the induction assumption.  Hence,  \eqref{isoper_huippu2} yields
\[
|A_{k+1}| \leq  (1-c_0)|A_{k}| \leq  (1-c_0)^{k+1 -k_1}|A_{k_1}|
\]
which proves \eqref{huippu_iteraatio3}.

\end{proof}

Theorem \ref{weakHarnack} follows from Lemma \ref{lemma_huippu2} and from the following result  which is  similar to the one in \cite{CC}.

\begin{lemma}
\label{caffarelli}
Let  $u \in C(B_2)$ be a supersolution of \eqref{model1.1} and a subsolution of  \eqref{model2.2} for $R \leq 1$ in $B_2$. Suppose that  $A_k \subset B_{\frac{5}{3}}$, $a_k \in \R$ and $L$ are as in Lemma \ref{lemma_perus}, and assume that  \eqref{summa-raja}  in   Lemma \ref{lemma_huippu2} holds for an  index $k_1$. Denote $m = \inf_{B_1} u$. There are universal constants  $L_0$ and $\sigma$ such that for $\nu= \frac{L_0}{L_0 - 1/2}$ the following holds: if $x_0 \in B_{4/3}$ and $l \in \N$ are such that 
\[
u(x_0) \geq \nu^{l}L_0 L^{k_1}m
\]
then it holds
\[
\sup_{B_{r_l}(x_0)} u >   \nu^{l+1}L_0 L^{k_1}m
\]
where  $r_l  = \sigma \nu^{- (l+1)\eps /n} (L_0/2)^{- \eps /n} a_{k_1}$. Here $\eps>0$ is from \eqref{epsilon_muoto}.
\end{lemma}

The  proof of the previous lemma  can be found in  Appendix. We now give the  proof of  the main result. 

\begin{proof}[Proof of Theorem \ref{weakHarnack}]
Let us assume that $u \in C(B_2)$ is a  supersolution of \eqref{model1.1} and a subsolution of \eqref{model2.2}  in $B_2$ for $R \leq 1$. By Remark \ref{triviaali_ska} we need to show that 
\[
\int_m^M \frac{dt}{R\phi(t)+ t} \leq C,
\]
where $m = \inf_{B_1}u$ and $M= \sup_{B_1}u$.  Since $a_k$ are given by \eqref{def.a_k} and $\phi$ is increasing  we have  
\begin{equation} \label{summa=int}
 \int_{m}^{L^k m}\frac{dt}{R\phi(t)+t} \leq  L \sum_{j=0}^{k-1} \frac{L^jm}{R\phi(L^jm )+ L^jm}  = L\sum_{j=0}^{k-1} a_j
\end{equation}
for every $k = 1,2,3, \dots$. Let $C$ be the constant from Lemma \ref{lemma_huippu2}. If 
\[
\sum_{j=0}^{\infty} a_j \leq C 
\]
the claim is trivially true. Let us treat the case when we have   \eqref{summa-raja}, i.e.,  there is an index $k_1$ such that 
\[
 \sum_{j=0}^{k_1-1} a_j \leq C
\]
where $C$ is a uniform constant and 
\[
|A_{k_1+l}| \leq (1-c_0) ^{l}a_{k_1}^n, \quad l =0,1,2,\dots.
\]
Suppose that $r_l = \sigma\nu^{- \eps l /n} L_0^{-\eps/n} a_{k_1}$ are as in Lemma \ref{caffarelli}. Then there is a uniform index $l_0$ such that 
\[
\sum_{j= l_0}^{\infty} r_j \leq 1/3.
\]
I claim that it holds 
\begin{equation} \label{final_claim}
\sup_{B_{1}} u \leq \nu^{l_0}L_0 L^{k_1}m.
\end{equation}
Indeed, if this were not true there would be a point $x_{l_0} \in B_{1}$ such that 
\[
u(x_{l_0}) \geq \nu^{l_0}L_0 L^{k_1}m.
\]
By Lemma \ref{caffarelli}  there is $x_{l_0+1} \in B_{r_{l_0+1}}(x_{l_0})$ such that 
\[
u(x_{l_0+1}) \geq \nu^{l_0+1}L_0 L^{k_1}m.
\]
We may repeat this process since at every step we have $|x_{j+1}- x_{j}| \leq r_j$ and therefore for every $l >l_0$ it holds
\[
|x_l| \leq |x_{l_0}|+ \sum_{j =l_0 }^{l-1}|x_{j+1}- x_{j}| < 1+  \sum_{j= l_0}^{\infty} r_j \leq \frac{4}{3}.
\]
Hence, $u(x_l) \geq  \nu^{l}L_0 L^{k_1}m$ and $x_l \in B_{4/3}$ for every $l > l_0$. This implies that $u$ is unbounded in $\bar{B}_{4/3}$  which contradicts  the continuity of $u$  in $B_2$.

From \eqref{summa=int} and \eqref{final_claim} we deduce 
\[
\begin{split}
 \int_{m}^{M}\frac{dt}{R\phi(t)+ t} &\leq  \int_{L^{k_1}m}^{ \nu^{l_0}L_0 L^{k_1}m} t^{-1}\, dt  +  \int_{m}^{L^{k_1}m}\frac{dt}{R\phi(t)+ t} \\
&\leq  \log \left(  \nu^{l_0} L_0 \right)+ L\sum_{j=0}^{k_1-1} a_j\\
&\leq \log \left(  \nu^{l_0} L_0 \right) + LC
\end{split}
\]
and the result follows. 
\end{proof}

We conclude the section with  a proof of Corollary \ref{holder}.

\begin{proof}[Proof of  Corollary \ref{holder}]
 Let us  show that if $v \in C(B_{R_0}(x_0))$  is a nonnegative  viscosity supersolution of \eqref{model1} and a  subsolution of \eqref{model2} in $B_{R_0}(x_0)$ such that $ \sup_ {B_{R_0}(x_0)} v  \leq M_0$, then there is $\hat{R} \leq R_0/2$ depending on $M_0$ such that for every $R \leq \hat{R}$ it holds
\begin{equation}
\label{harnack_helppo}
\sup_{B_R}v \leq C(\inf_{B_R}v + \sqrt{R})
\end{equation}
for a uniform constant  $C$. The H\"older continuity of $u$ then follows from \eqref{harnack_helppo} by  a standard iteration argument (\cite[Chapter 8.9]{GT}). 

To show \eqref{harnack_helppo} we denote $m = \inf_{B_R}v$ and $M= \sup_{B_R}v$.  By Proposition \ref{prop_hitaasti} there is $\hat{R}$ such that for every $R \leq \hat{R}$ it holds
\[
 \sqrt{R}\,  \eta \left( \frac{M_0}{R} \right)  \leq \frac{1}{M_0}.
\]
Therefore for  every $R\leq \hat{R}$ and  $t \leq M_0$ we have 
\[
R^2\phi(t/R) \leq R^2\phi(M_0/R) = M_0  R  \, \eta \left( \frac{M_0}{R} \right)  \leq \sqrt{R}.
\]
Since $M \leq M_0$ we may simply estimate 
\[
  \int_{m}^{M} \frac{dt}{R^2\phi(\frac{t}{R})+ t} \geq \int_{m}^{M} \frac{dt}{\sqrt{R}+ t} =\log\left( \frac{M + \sqrt{R}}{m+ \sqrt{R}}\right).
\]
From Theorem \ref{weakHarnack} we deduce
\[
M \leq C(m+ \sqrt{R}).
\]
Hence we have  \eqref{harnack_helppo}.
\end{proof}

\section{On $p(x)$-harmonic functions}

In this section we  discuss how Theorem \ref{weakHarnack} implies   Corollary \ref{p(x)-harnack}. Moreover, we will see that this inequality is optimal. Let us recall that  a function $u \in W_{loc}^{1,1}(\Omega)$ is  $p(x)$-harmonic in $\Omega$  if it is local minimizers of the  energy
\[
\int_{\Omega} \frac{1}{p(x)} |Du|^{p(x)}\, dx,
\]
where $1 < p(x)< \infty$.  We assume  that $p \in C^1(\R^n)$ and that there are numbers $1 <p^- \leq p^+ < \infty$ such that $p^- \leq p(x) \leq p^+$ for every $x \in \R^n$.   For more about  $p(x)$-harmonic functions see \cite{AM} and the references therein.

It follows from \cite{AM} that under these conditions $p(x)$-harmonic functions are locally $C^{1,\alpha}$-regular.  In \cite{JLP} it was shown that $p(x)$-harmonic functions are viscosity solutions of the Euler-Lagrange equation in nondivergence form. In order to formulate this result more precisely we define the following operator
\[
\Delta_{p(x)}\vphi(x):=  \Delta \vphi(x) + (p(x) -2) \Delta_{\infty}\vphi(x) +  \log|D\vphi(x)| \langle Dp(x), D\vphi(x) \rangle , 
\]
where  $\Delta_{\infty}\vphi= \langle D^2\vphi \frac{D\vphi}{|D\vphi|}, \frac{D\vphi}{|D\vphi|} \rangle$ denotes the infinity Laplace operator. This operator is well defined whenever $D\vphi(x) \neq 0$. The following result is from \cite{JLP}. 
\begin{proposition}
\label{px_visco}
Let $u \in C(\Omega) \cap  W_{loc}^{1,1}(\Omega)$ be $p(x)$-harmonic in $\Omega$. If $\vphi \in C^2(\Omega)$ is  such that $\vphi(x_0)= u(x_0)$ at $x_0 \in \Omega$, $D\vphi(x_0) \neq 0$ and $\vphi \leq u$  then it holds
\[
- \Delta_{p(x)}\vphi(x_0) \geq  0,
\]
and if $\vphi \geq u$ then 
\[
- \Delta_{p(x)}\vphi(x_0) \leq 0.
\]
\end{proposition}

Corollary  \ref{p(x)-harnack} follows immediately from Theorem \ref{weakHarnack} once we  show that  $p(x)$-harmonic functions are viscosity supersolutions of \eqref{model1} and subsolutions of \eqref{model2} for $\phi(t)= C(|\log t|+1)t$ for some $C$. This is the assertion of the next lemma.
\begin{lemma}
\label{px_lemma}
Let $u \in C(\Omega)  \cap W_{loc}^{1,1}(\Omega)$ be $p(x)$-harmonic in $\Omega$ and let  $\phi(t)= C(|\log t|+1)t$  where  $C = ||p||_{C^1(\Omega)} < \infty$. If $\vphi \in C^2(\Omega)$ is such that $\vphi \leq u$ and $\vphi(x_0)= u(x_0)$ at $x_0 \in \Omega$  then it holds
\[
\Pu_{\lambda, \Lambda}^+(D^2 \vphi(x_0))  \geq -\phi(|D\vphi(x_0)|),  
\]
i.e., it is a viscosity supersolution of  \eqref{model1}, and if $\vphi \geq u$  then it holds
\[
\Pu_{\lambda, \Lambda}^-(D^2  \vphi(x_0))  \leq \phi(|D \vphi(x_0)|), 
\]
i.e., it is a viscosity subsolution of  \eqref{model2}. Here $\lambda = \min\{ 1, p^- -1\}$ and $\Lambda = \max \{ 1, p^+ -1 \}$.
\end{lemma}

\begin{proof}
We only prove that $u$ is  a viscosity supersolution of  \eqref{model1}, for the subsolution property follows  similarly. Let $\vphi \in C^2(\Omega)$ be  such that $\vphi(x_0)= u(x_0)$ at $x_0 \in \Omega$ and $\vphi \leq u$ in a neighborhood of $x_0$. Without loss of generaly we may assume that $x_0 = 0$, $u(0)=0$,  $\vphi(x) = \langle Ax, x \rangle + \langle b, x \rangle$ for a  symmetric matrix  $A$ and a vector $b$, and that $\vphi(x) < u(x)$ for $x \neq 0$  in $B_\rho$ for some small $\rho >0$. The goal is to show that 
\[
\Pu_{\lambda, \Lambda}^+(D^2 \vphi(0))  +\phi(|D\vphi(0)|) \geq 0.
\]
Note that if $D\vphi(0) \neq 0$ then the claim follows from Proposition \ref{px_visco} since
\[
\Pu_{\lambda, \Lambda}^+(D^2 \vphi(0)) +   \phi(|D\vphi(0)|) \geq - \Delta_{p(x)}\vphi(0) \geq  0.
\]
Therefore we need to treat the case $D\vphi(0) =b = 0$ to conclude the proof. 

Let $r>0$ be small. For $y \in B_r$ we define 
\[
\vphi_y(x)= \vphi(x-y).
\] 
For every  $y$ there is a number $c_y $  such that the function $\vphi_{y} +c_y$ touches $u$ from below, say at a point   $x_y$. It is clear that $x_y \to 0$ as $|y| \to 0$. If there exists a sequence $(y_k)$ such that $|y_k| \to 0$ and at  the associated contact points $(x_k)$ it holds $D \vphi(x_k)\neq 0$, Proposition \ref{px_visco} implies 
\[
\Pu_{\lambda, \Lambda}^+(D^2 \vphi(x_k)) +   \phi(|D\vphi(x_k)|) \geq - \Delta_{p(x)}\vphi(x_k) \geq  0.
\] 
The claim then follows  by letting $k \to \infty$. Hence,  we need to  treat the case when there exists $r>0$ such that for every $y \in B_r$  at every  associated contact point $x_y$ it holds $D\vphi_y(x_y) = 0$.

Since $\vphi_y(x) = \vphi(x-y) = \langle A(x-y), (x-y) \rangle + c_y$ and $D\vphi_y(x_y) = 0$ we have that $x_y = y$ for every $y \in B_r$. This means that at every point $y \in B_r$ we may touch the graph of $u$ from below with a paraboloid 
\[
P(x)= - |A| |x-y|^2 + u(y). 
\]
This implies that $u$ is semi-convex in $B_r$. In particular, $u$ is locally Lipschitz continuous in  $B_r$ and its gradient  vanishes almost everywhere.  Hence, $u$ is constant in $B_r$ and the claim is trivially true. 
\end{proof}

 We conclude this section by constructing a naive example which  verifies that  Corollary  \ref{p(x)-harnack} is  indeed sharp. To that aim let us denote the interval $I_r(k)= (k-r,k+r)$. We consider the function $u : (0, \infty)\to \R$,
\[
u(x)  = e^{-e^{x}}. 
\] 
It is easy to see that if $p_k \in C^1(I_2(k))$ solves the equation
\begin{equation} \label{tyhma_ody}
p'_k(x) +\left( \frac{e^x-1}{e^x-x}\right)(p_k(x)-1) =0 \quad x \in I_2(k)
\end{equation}
then $u$ is   $p_k(x)$-harmonic in $I_2(k)$. We may solve \eqref{tyhma_ody} such that  $p_k$ satisfies $1+ C^{-1} \leq p_k \leq C$ and $||p_k||_{C^1(I_2(k))}\leq C$ for a constant $C$ which is independent of $k$.

Note that the function $u$ does  not satisfy the classical Harnack's inequality, since
\[
\frac{\sup_{I_1(k)}u}{\inf_{ I_1(k)}u} =  \frac{e^{-e^{k-1}}}{e^{-e^{k+1}}} = e^{e^{k}( e- e^{-1})} \to \infty \qquad \text{as}\,\, k \to \infty.
\]
On the other hand Corollary  \ref{p(x)-harnack} implies 
\[
\sup_{I_1(k)}u^C \leq C \inf_{I_1(k)}u
\]
for a constant $C>1$ which is the optimal estimate.

\appendix

\section{Proof of the Lemmata of Section 4}
\label{appendix}

\begin{proof}[Proof of  Lemma \ref{barrior}] 
The construction of the barrier function is standard \cite{CC}. For every $x \in B_{2r_0} \setminus  B_{\frac{r_0}{2}}$ the function is defined as
\[
\vphi(x)= M_1 - M_2|x|^{- \alpha}
\]
where $\alpha =\max \{\frac{(n-1)\Lambda}{\lambda},1\}$ and $M_1, M_2$ are such that $\vphi(x) = 0$  when $|x| = 2r_0$ and  $\vphi(x) = -2$  when $|x| = r_0$. In other words
\[
M_2 =  \frac{2 r_0^{\alpha}}{1 - 2^{- \alpha}}.
\] 
Note that while $\alpha$ is already fixed, the radius  $r_0$ is still to be chosen. If we can  show that there is  $r_0$ such that 
\begin{equation} \label{este_yhtalo}
\Pu_{\lambda, \Lambda}^-(D^2 \vphi(x)) \geq \phi(|D\vphi(x)|)  \qquad x \in B_{2r_0} \setminus  B_{\frac{r_0}{2}}
\end{equation}
we may extend $\phi$ smoothly to the whole ball  $B_{2r_0}$ so that it will satisfy all the required conditions.

Let us find $r_0$ which satisfies  \eqref{este_yhtalo}. For $\frac{r_0}{2} \leq |x| \leq 2 r_0$, we have
\[
D \vphi(x) = \alpha M_2|x|^{- \alpha-2}x \quad \text{and} \quad D^2 \vphi(x) = \alpha M_2|x|^{- \alpha-2}\left(I - (\alpha +2) \frac{x}{|x|} \otimes \frac{x}{|x|} \right).
\]
Therefore it holds
\[
\begin{split}
\Pu_{\lambda, \Lambda}^-(D^2 \vphi) &\geq \alpha M_2|x|^{- \alpha-2} (\lambda(\alpha +1) - \Lambda (n-1))\\
&\geq   \alpha  \lambda  M_2 |x|^{- \alpha-2}\\
&\geq   \frac{\alpha \lambda}{2(2^{\alpha} -1)}r_0^{-2}
\end{split}
\]
by the choices of $\alpha$ and $M_2$. For $\frac{r_0}{2} \leq |x| \leq 2 r_0$ the monotonicity of $\phi$   yields 
\[
\phi(|D \vphi(x)|  ) = \phi \left(\frac{ 2 \alpha  r_0^{\alpha}}{1 - 2^{- \alpha}}|x|^{-\alpha-1}  \right) \leq  \phi \left( 2^{\alpha +3} \alpha r_0^{-1} \right).
\]
Therefore in order to show \eqref{este_yhtalo} we only need to find $r_0$ which satisfies 
\begin{equation} \label{este_lemma}
\frac{ \alpha \lambda}{2(2^{\alpha} -1)} r_0^{-2} \geq \phi \left( 2^{\alpha +3} \alpha r_0^{-1} \right).
\end{equation}

Writing $\phi(t)= \eta(t)t$  \eqref{este_lemma} reads as
\[
\frac{\lambda}{2^{\alpha +4}(2^{\alpha} -1)}r_0^{-1} \geq \eta \left( 2^{\alpha +3} \alpha r_0^{-1} \right).
\]
By Proposition  \ref{prop_hitaasti} (ii) we have  
\[
\lim_{t \to \infty} \frac{\eta(t)}{t} =0
\]
and therefore \eqref{este_lemma} holds for  $r_0$ small enough. 
\end{proof}

\begin{proof}[Proof of  Lemma \ref{decay_est}]
By approximating $u$ with infimal convolution
\[
u_{\eps}(x) = \inf_{y \in B_2}\left( u(y) + \frac{1}{\eps}|x-y|^2\right) 
\]
we may assume that $u$ is semiconcave.

Let $\hat{x} \in \bar{B}_{1}$ be a point where $u(\hat{x})\leq 1$. Let $r_0 \leq 1$ be as in Lemma \ref{barrior} and choose $x_0 \in B_1$ such that $\hat{x} \in \bar{B}_{r_0}(x_0) \subset \bar{B}_1$. Let  $\vphi$ be the barrier function from Lemma \ref{barrior} and define $v:B_{2r_0} \to  \R$ as
\[
v(x) = u(x) + \vphi(x-x_0). 
\]
By Lemma \ref{barrior} (ii) we have  $\inf_{B_{r_0}(x_0)} v \leq v(\hat{x} ) \leq -1$. 

Since $u$ is nonnegative   and  $\vphi(x-x_0)\geq 0$ for $x \in \partial  B_{2r_0}(x_0)$   (Lemma \ref{barrior} (i)), we have $v \geq 0 $ on $\partial  B_{2r_0}(x_0)$.  Moreover, by the monotonicity of $\phi$, by elementary properties of the Pucci-operators and by  Lemma  \ref{barrior} (iii) we obtain that $v$ is a viscosity supersolution of
\begin{equation} \label{from_pucci}
\Pu^+(D^2v (x)) \geq - \phi(|Dv(x)| + |D \vphi(x-x_0)|) + \phi(|D \vphi(x-x_0)|) - C\xi(x-x_0),
\end{equation}
in $ B_{2r_0}(x_0)$. Here $\xi$ is a continuous function  such that $0 \leq \xi \leq 1$ and $\text{supp} \,\xi  \subset B_{\frac{r_0}{2}}$.

Let us extend $v$  by $0$  outside   $B_{2r_0}(x_0) $ and denote the convex envelope of $-v^-= \min\{v, 0\}$ in $B_{3r_0}(x_0) $  by $\Gamma_v$, i.e., 
\[
\Gamma_v(x):= \sup_{p \in \R^n} \inf_{y \in B_{3r_0}(x_0) } ( p \cdot (x-y)  - v^-(y) ).
\]
We  denote the contact set  by $\{ v = \Gamma_v \} := \{ x \in B_{3r_0}(x_0) : -v^-(x) = \Gamma_v(x)\}$.  Since $v$ is semiconcave we have $\Gamma_v \in C^{1,1}(\{ v = \Gamma_v \})$, see \cite[Theorem 5.1]{CC}. In particular, $\Gamma_v $ is twice differentiable almost everywhere on $\{ v = \Gamma_v \}$. Since $v$ is a viscosity supersolution of \eqref{from_pucci} and $\Gamma_v \leq v$ in $B_{2r_0}(x_0) $ we have
\begin{equation} \label{from_pucci2}
\Pu^+(D^2 \Gamma_v (x)) \geq - \phi(|D  \Gamma_v(x)| + |D \vphi(x-x_0)|) + \phi(|D \vphi(x-x_0)|) - C\xi(x-x_0)
\end{equation}
 a.e. $x \in  \{ v = \Gamma_v \}$.

We denote by $E$ the subset of $B_{2 r_0}(x_0) \cap \{ v = \Gamma_v \}$ where the gradient of $\Gamma_v$ is less than one 
\begin{equation} \label{set_E}
E :=  \{ v = \Gamma_v \}\cap \{ x \in B_{2 r_0}(x_0)  : |D \Gamma_v(x)| \leq 1\}.
\end{equation}
 I claim that it holds
\begin{equation} \label{eq_for_v}
\Pu^+(D^2 \Gamma_v(x)) \geq - b|D\Gamma_v(x)|  - C_1\chi_{B_{r_0}}(x)
\end{equation}
a.e. on $E$, for some universal constants  $C_1, b$, where $\chi_{B_{r_0}}$ is the characteristic function of $B_{r_0}$.

Indeed, denote $\tilde{L}:= \sup_{x \in  B_{2r_0} } |D \vphi(x)|$. By  Lemma \ref{barrior} it holds  $ |D \vphi(x)| \geq L_1^{-1}$ for $x \in  B_{2r_0} \setminus B_{\frac{r_0}{2}}$. Since $|D\Gamma_v(x)| \leq 1$  for $x \in E$, we have by the local Lipschitz continuity of $\phi$ that 
\begin{equation} \label{lokaali_lip}
 \phi(|D\Gamma_v(x)| + |D \vphi(x-x_0)|) - \phi(|D \vphi(x-x_0)|) \leq b \, |D \Gamma_v(x)| 
\end{equation}
a.e. on $E \setminus B_{\frac{r_0}{2}}(x_0)$, where 
\[
b = \max \{|D\phi(p)| \, : \,  L_1^{-1} \leq |p| \leq \tilde{L} +1\}.
\] 
Since $\text{supp} \, \xi  \subset B_{\frac{r_0}{2}}$ we conclude from \eqref{from_pucci2} and \eqref{lokaali_lip} that
\[
\Pu^+(D^2 \Gamma_v(x)) \geq - b|D  \Gamma_v(x)| 
\]
a.e. on $E \setminus B_{\frac{r_0}{2}}(x_0)$. On the other hand, we may  trivially estimate from \eqref{from_pucci2} that 
\[
\begin{split}
\Pu^+(D^2  \Gamma_v (x)) &\geq - \phi(|D  \Gamma_v(x)| + |D \vphi(x-x_0)|) + \phi(|D \vphi(x-x_0)|) - C\xi(x-x_0) \\
&\geq -\phi(1+ \tilde{L}) -C
\end{split}
\]
a.e. on $E  \cap B_{\frac{r_0}{2}}(x_0)$. Hence we have \eqref{eq_for_v}.

Next we use \eqref{eq_for_v} to deduce  
\[
0 \leq \det(D^2 \Gamma_v(x)) \leq C_2( |D \Gamma_v(x)|^n + \chi_{B_{r_0}}(x))
\] 
a.e. $x \in E$ for some universal  constant $C_2$.  The previous inequality and  the coarea formula yield
\begin{equation} \label{abp-estimate}
\begin{split}
\int_{D \Gamma_v(E)} \frac{dp}{|p|^n + \delta} &\leq \int_{E} \frac{\det(D^2 \Gamma_v)}{|D \Gamma_v|^n + \delta}\, dx \\
&\leq C_2 \int_{E} \frac{|D \Gamma_v|^n + \chi_{B_{r_0}}}{|D \Gamma_v|^n + \delta}\, dx \\
&\leq C_2\,|B_{2r_0}|  + \frac{C_2}{\delta} | B_{r_0}(x_0) \cap \{ v = \Gamma_v \}|,
\end{split}
\end{equation}
where $\delta>0$ is a small number which will be chosen later. 

Let us recall that $v = 0$ in $B_{3r_0}(x_0) \setminus B_{2r_0}(x_0) $ and $\inf_{B_{r_0}(x_0)}v \leq -1$.  I claim that it holds 
\begin{equation} \label{abp-estimate2}
B_{1/4} \subset  D \Gamma_v(E),
\end{equation}
where the set $E$ is defined in \eqref{set_E}. Indeed, let us choose $p \in B_{1/4}$. The function $w(x)= -v^-(x) -p \cdot (x-x_0) +3r_0|p|$ is nonnegative on $\partial B_{3r_0}(x_0)$ and $\inf_{B_{r_0}(x_0)} w  <0$.  Therefore $w$ attains its  minimum in $B_{3r_0}(x_0)$, say at $\tilde{x}$. In particular, $\tilde{x}$ belongs to the contact set $\{\Gamma_v =v\}$ and $p=D \Gamma_v(\tilde{x})$. Note that  $v(\tilde{x})<0$ and therefore $\tilde{x} \in B_{2r_0}(x_0) $, since $v = 0$ in $B_{3r_0}(x_0) \setminus B_{2r_0}(x_0) $. Moreover we have $|D \Gamma_v(\tilde{x})| = |p| < 1/4$. Hence,   $\tilde{x} \in E$ which proves  \eqref{abp-estimate2}.

The estimate \eqref{abp-estimate2} yields 
\[
\int_{D \Gamma_v(E)} \frac{dp}{|p|^n + \delta} \geq \int_{B_{1/4}} \frac{dp}{|p|^n + \delta} = \omega_n \int_{0}^{1/4}\frac{\rho^{n-1}}{\rho^n + \delta}. 
\]
Since the value of the integral diverges as $\delta \to 0$, we may choose $\delta$ such that 
\[
\omega_n \int_{0}^{1/4}\frac{\rho^{n-1}}{\rho^n + \delta} \geq  C_2\,|B_{2r_0}| +1.
\]
The estimate \eqref{abp-estimate} then implies
\[
|  B_{r_0}(x_0) \cap \{ v = \Gamma_v \}| > \mu. 
\]
for some $\mu>0$.

To conclude the proof we notice that  if $x$ belongs to the  contact set $ \{v = \Gamma_v\}$, then it  holds  $v(x) \leq 0$. Therefore $u(x) \leq- \vphi(x- x_0) \leq L_1$ for every  $x \in B_{r_0} \cap \{ v = \Gamma_v \} $ and the previous inequality yields
\[
|\{ u(x) \leq L_1 :  x \in B_{r_0}(x_0) \}| > \mu.
\]
The claim follows from $ B_{r_0}(x_0) \subset B_1$.
\end{proof}

The proof of Lemma \ref{caffarelli}  is again standard  \cite{CC} except that we need to be careful when we rescaling the solution. Recall that  there is no bound for the index $k_1$ and no assumption of the infimum $m = \inf_{B_1}u$.   
\begin{proof}[Proof of Lemma \ref{caffarelli}]
Let $\mu$ be the constant from Lemma \ref{decay_est}. Let us choose $\sigma^n > \frac{c_{\eps}}{ \omega_n \mu}$, where $c_{\eps}$ is the constant from \eqref{epsilon_muoto} and $\omega_n$ is the volume of the unit ball. 

We argue  by contradiction and assume  that $\sup_{B_{2r_l}(x_0)} u \leq  \nu^{l+1}L_0 L^{k_1}m$. Recall that $r_l  = \sigma \nu^{- (l+1)\eps /n} (L_0/2)^{- \eps /n} a_{k_1}$ and $a_{k_1}= \frac{1}{L} \frac{1}{R \eta(L^{k_1}m) +1}$.  By choosing $L_0$ large enough we may assume that $r_l  \leq \frac{1}{6}$. The estimate  \eqref{epsilon_muoto} from  Lemma \ref{lemma_huippu2} gives
\begin{equation} 
\label{decay_ylh}
\begin{split}
|\{ x \in B_{r_l}(x_0) :  u(x) \geq   \nu^{l+1} \frac{L_0}{2} L^{k_1}m \}| &\leq  |\{ x \in B_{\frac{5}{3}} :  u(x) \geq  \nu^{l+1} \frac{L_0}{2} L^{k_1}m \}|\\
 &\leq  c_{\eps}  \nu^{-(l+1)\eps} \left(\frac{L_0}{2}\right)^{-\eps} a_{k_1}^n.
\end{split}
\end{equation}

 We  define a nonnegative function $v: B_2 \to \R$ by
\[
v(x):= \frac{\nu}{(\nu-1)} - \frac{u(r_lx +x_0)}{(\nu-1)\nu^{l}L_0 L^{k_1}m }.
\]
Denote $A= (\nu-1) \nu^{l}L_0 L^{k_1}m$. I claim that it holds
\[
r_l \leq \frac{A}{R \phi(A) + A} = \frac{1}{R\eta( (\nu-1) \nu^{l}L_0 L^{k_1}m)+ 1}.
\]
Indeed, this is  equivalent to 
\begin{equation} 
\label{alh_skaalaus}
R\eta( (\nu-1) \nu^{l}L_0 L^{k_1}m) +1 \leq \sigma^{-1}\nu^{ (l+1)\eps /n} (L_0/2)^{ \eps /n}  a_{k_1}^{-1}.
\end{equation}
Since  $a_{k_1}^{-1}  \geq R\eta( L^{k_1}m) +1$ we have  by the condition (P3), by $\eta\geq 1$ and by the choice of $\nu$ that
\[
\begin{split}
R\eta( (\nu-1) \nu^{l}L_0 L^{k_1}m)+ 1 &\leq \Lambda_0 \eta( (\nu-1) \nu^{l}L_0)(R\eta( L^{k_1}m)+1) \\
&\leq C \eta( \nu^{l}) a_{k_1}^{-1}.
\end{split}
\]
By Proposition \ref{prop_hitaasti} (ii) it holds $\lim_{t \to \infty}\frac{\eta(t)}{t^{\eps/n}}= 0$ and therefore   we have 
\[
\sigma^{-1}\geq \frac{ C }{ (L_0/2)^{ \eps /n}} \left(\sup_{l\geq 1}\frac{ \eta(  \nu^{l}) }{\nu^{ (l+1)\eps /n}} \right)
\]
when $L_0$ is chosen large enough. This proves \eqref{alh_skaalaus}.

By Lemma \ref{skaalaus} we deduce that $v$ is a nonnegative supersolution of \eqref{model1} in $B_2$ and $v(0)=1$. Hence, Lemma \ref{decay_est} yields
\begin{equation} 
\label{decay_alh}
\begin{split}
|\{ x \in B_{r_l}(x_0) :  u(x) < \nu^{l+1} \frac{L_0}{2} L^{k_1}m \}| &= |B_{r_l}| |\{ x \in B_{1}(x_0) :  v (x) >  L_0 \}| \\
&\leq (1-\mu) |B_{r_l}|
\end{split}
\end{equation}
when $L_0 \geq L$. Combining \eqref{decay_ylh} and \eqref{decay_alh} we have
\[
\begin{split}
|B_{r_l}| &= |\{ x \in B_{r_l}(x_0) :  u(x) >  \nu^{l+1}  \frac{L_0}{2}  L^{k_1}m\}|  + |\{ x \in B_{r_l}(x_0) :  u(x) \leq   \nu^{l+1}  \frac{L_0}{2}  L^{k_1}m\}|  \\
&\leq   c _{\eps} \nu^{-(l+1)\eps} \left(\frac{L_0}{2}\right)^{-\eps}a_{k_1}^n+ (1-\mu) |B_{r_l}|. 
\end{split}
\]
In other words
\[
\mu  \omega_n r_l^n  \leq c _{\eps}    \nu^{-(l+1)\eps} \left(\frac{L_0}{2}\right)^{-\eps} a_{k_1}^n.  
\]
Since  $r_l  = \sigma \nu^{- (l+1)\eps /n} (L_0/2)^{- \eps /n} a_{k_1}$ this implies 
\[
 \sigma^n \leq \frac{c _{\eps} }{\omega_n \mu}
\]
which  contradicts the choice of $\sigma$.
\end{proof}

\section*{Acknowledgment}
 This  work was supported by the  Academy of Finland grant 268393.

\end{document}